\documentclass{amsart}

\usepackage[english]{babel}
\usepackage{amsfonts, amsmath, amsthm, amssymb,amscd,indentfirst}

\usepackage{mathrsfs}

\usepackage{tikz-cd}

\usepackage{scalerel}
\usepackage{stackengine,wasysym}

\usepackage{color}
\usepackage{xcolor}
\usepackage{tikz}
\usepackage{esint}

\usepackage{esint}
\usepackage{graphicx}
\usepackage{epstopdf}
\usepackage{amsmath,amssymb,latexsym,indentfirst}
\usepackage{times}
\usepackage{palatino}

\usepackage{stackengine}
\usepackage{scalerel}

\def\ba{\begin{align}}
	\def\ea{\end{align}}
\def\bp{\begin{proof}}
	\def\ep{\end{proof}}

\theoremstyle{plain}
\newtheorem{theorem}{Theorem}[section]

\newtheorem{proposition}[theorem]{Proposition}

\theoremstyle{definition}
\newtheorem{definition}[theorem]{Definition}

\newtheorem{remark}[theorem]{Remark}

	\def\bp{\begin{proof}}
		\def\ep{\end{proof}}

	\def\dirac{{\eth}}
	\def\nonind{{\mathfrak I\!\!\!/}}

\begin{document}

		\title[The scalar curvature in wedge spaces]{The scalar curvature in wedge spaces: existence and obstructions}

		\author{Levi Lopes de Lima}
		\address{Universidade Federal do Cear\'a (UFC),
			Departamento de Matem\'{a}tica, Campus do Pici, Av. Humberto Monte, s/n, Bloco
			914, 60455-760,
			Fortaleza, CE, Brazil.}
		\email{levi@mat.ufc.br}
		\thanks{
			L.L. de Lima has been supported by 
			FUNCAP/CNPq/PRONEX 00068.01.00/15.}

		\begin{abstract}
			We study the scalar curvature of incomplete wedge metrics in certain stratified spaces with a single singular stratum (wedge spaces). Building upon several well established technical tools for this category of spaces (the corresponding Yamabe, elliptic and index theories) we 
			provide existence and obstruction results for such metrics under suitable positivity assumptions on the underlying geometry. This is meant to be a follow-up to a previous paper of ours (AGAG, 2022),  where the case of spaces with an isolated conical singularity was considered.
		\end{abstract}

		\maketitle

		\tableofcontents

		\section{Introduction}\label{intro}
		
		The general problem of prescribing the scalar curvature function in a given smooth closed manifold is a central theme in Riemannian Geometry. Since in principle this metric invariant only affects the underlying geometry at a local level with no direct influence on its large scale behavior, it is expected that a huge amount of functions might be realized as the scalar curvature. In a sense this has been confirmed by Kazdan and Warner \cite{kazdan1975existence}, who showed that any function on a closed manifold of dimension $n\geq 3$ which is negative somewhere is the scalar curvature of some metric. Versions of this result in  the setting of compact manifolds with boundary have been established in \cite{cruz2019prescribing}. These contributions should be contrasted with the well-known topological obstructions to the existence of metrics with positive scalar curvature in the spin setting stemming from the works of Lichnerowicz \cite{lichnerowicz1963spineurs}, Hitchin \cite{hitchin1974harmonic} and Gromov-Lawson \cite{gromov1980spin,gromov1996positive} and relying upon the remarkable properties of the Dirac operator acting on spinors; see also \cite{bar2023boundary} for similar obstructions in the presence of a boundary.
		
		The analysis in \cite{de2022scalar} has suitably extended the aforementioned results to the category of spaces with an isolated conical singularity. The purpose of this note is to investigate further extensions of these contributions to the more general category of (incomplete) edge spaces with a single singular stratum, also named {\em wedge spaces}. The main results here are Theorems \ref{presc:ex} and \ref{obs:res:karea} dealing respectively with existence and obstructions of wedge metrics with prescribed scalar curvature on such spaces under certain positivity assumptions on the underlying geometry. In particular, we provide a unified treatment that retrieves the conical case studied in \cite{de2022scalar}. Also, as explained in Remark \ref{rosen}, the obstruction in Theorem \ref{obs:res:karea} in a sense complements those obtained in \cite{albin2016index,botv2021positive} by using similar tools (the wedge index theory in \cite{albin2016index}).
		
		In order to carry out the arguments needed to establish Theorem \ref{presc:ex}, certain aspects of the solution of the Yamabe problem in this wedge category are reviewed  in Section \ref{yamabe:ie}, following \cite{akutagawa2014yamabe} closely. 
		Another key ingredient here is the mapping theory of geometric elliptic operators on such spaces (specifically, the Laplacian acting on functions and the Dirac operator acting on spinors), which is described in Sections \ref{wedge:ell:th} and \ref{top:obst}; see, among others,  \cite{mazzeo1991elliptic,melrose1993atiyah,lesch1997differential,schulze1998boundary,grieser2001basics,egorov2012pseudo} for the general theory  and also \cite{de2022mapping} for an informal account of this topic in the conical setting. 
		This theory is used in Section \ref{wedge:ell:th} to establish good mapping properties for the scalar Laplacian. More precisely, we check here that in all cases of interest this operator is essentially self-adjoint when viewed as an unbounded and symmetric operator acting on a certain weighted Sobolev space, with the corresponding self-adjoint extension being Fredholm (Theorem \ref{self:ad:lap}). This justifies the integration by parts needed to carry out the perturbative scheme leading to Theorem \ref{presc:ex}. We remark that surjectivity would suffice in this part of the argument, but we prefer to establish the finer mapping properties for at least two reasons:  they provide a local description of the space of solutions of an associated non-linear problem (see Remark \ref{semi:lin}) and, more fundamentally, the corresponding analysis may be easily transplanted to the Dirac setting, where
		surjectivity does not suffice and both self-adjointness and Fredholmness are crucial to the proof of Theorem \ref{obs:res:karea}.

		The mapping properties mentioned above should routinely follow from results in the existing literature (as in the accounts for the Hodge Laplacian in \cite{mazzeo2012analytic,albin2022cheeger}, for instance) but since we were unable to locate a specific source carrying out the analysis for the scalar Laplacian, we supply a fairly detailed guide to the quite involved proof, as this also will allow us to keep track of the rather subtle role (the dimension of) the link manifold plays in achieving self-adjointness; see Remarks \ref{donald:2} and \ref{f:geq:3}. Among the various routes available, we have chosen to adopt as a key input here an argument employed in  \cite{albin2012signature,albin2018hodge,albin2016index} for Dirac-type operators and in \cite{albin2022cheeger} for the Hodge Laplacian, both relying on certain regularity results 
		in \cite{gil2003adjoints,gil2013closure}. This same reasoning is used in Section \ref{top:obst} to probe the mapping properties of the Dirac operator needed in the proof of Theorem \ref{obs:res:karea}. 
		Besides its relevance to the applications mentioned above, we believe that the approach we take here has an independent interest as it might be implemented in other non-linear geometric problems where Laplace-type operators show up at an infinitesimal level. 
		
		We now briefly describe the strategy behind the proofs of our results.  Regarding Theorem \ref{presc:ex}, we adapt the argument laid down in \cite{kazdan1975existence} in the smooth setting. The first step is to find a wedge metric with constant negative (say $-1$) scalar curvature (Theorem \ref{const:neg}). In the smooth category, this corresponds to the ``easy'' case of the solution of the classical Yamabe problem \cite{leeparker1987}, hence always solvable. Unfortunately, finding such a metric in the wedge category with the available technology 
		(the singular Yamabe theory in \cite{akutagawa2014yamabe}) turns out to be considerably much harder and seems to require some kind of positivity condition on the geometry of the link manifold. The simplest possibility, which suffices for the applications we have in mind, is to assume that the link metric has constant positive scalar curvature (the more general case treated in Theorem \ref{presc:ex}, which merely assumes that the link metric is Yamabe positive, may be reduced to this simpler case in view of the well-known solution of the Yamabe problem for closed manifolds \cite{schoen1984conformal,leeparker1987}; see Remark \ref{yam:pos}). In any case, with this wedge metric  (say $g$) at hand, the next step consists in showing that any function ``close'' enough to $-1$ may be realized as the scalar function of some conformal wedge metric (if this is the case, it is straightforward to check that the natural action of the group of diffeomorphisms of bounded distortion on the space of wedge metrics takes care of the general case, in which the prescribed scalar curvature function is negative somewhere; see Proposition \ref{conf:dist}). This step involves a perturbative argument relying on the invertibility of the Laplace-type operator  $-\Delta_g+\lambda$, $\lambda>0$, arising as the linerization at $g$ of the associated non-linear problem. Again, in the smooth category the invertibility of this operator in the standard Sobolev scale is well established. In our case, however, one is led to work in a Sobolev scale that takes into account the structure of the singular stratum  and it is precisely here that the mapping theory mentioned above is needed. The outcome of applying this theory to our problem appears in Theorem \ref{self:ad:lap}, confirming that in most cases the Laplacian has good mapping properties, which leads to Theorem \ref{presc:ex}. Finally, we mention that this same mapping theory is also needed to establish self-adjointness and Fredholmness for the Dirac operator (Theorem \ref{dirac:map:p}), which enables the use of the index theory from \cite{albin2016index} employed in the proof of Theorem
		\ref{obs:res:karea}.     
		
		\vspace{0.2cm}
		\noindent{\bf Acknowledgments.} The author would like to thank S. Almaraz for helpful discussions and comments.

		\section{Wedge spaces and statements of the main results}\label{setup}
		
		We start by describing the class of spaces we are interested in. For simplicity, all smooth manifolds appearing below are assumed to be oriented.
		
		\begin{definition}\label{ie:cat}
			A stratified space with a single singular stratum is a topological space $X$ satisfying:
			\begin{enumerate}
				\item $X$ is a (compact and connected) metric space with distance function $\mathfrak d$;
				\item  $X=Y \sqcup X_s$, where $X_s$ is a (open and dense) smooth manifold of dimension $n\geq 3$ (the smooth stratum of $X$);
				\item $Y$ is a (connected and boundaryless) smooth manifold of dimension $b\geq 0$ (the singular stratum);
				\item there exists a (connected) neighborhood $U$ of $Y$ (the wedge region) such that: 
				\begin{itemize}
					\item $X\backslash U$ is a smooth manifold with boundary $Y^\bullet$;
					\item there exists a retraction $\pi_Y: U\to Y$ with $\pi_Y|_{U\backslash Y}$ being smooth;
				\end{itemize}
				\item there exists a ``radial function'' $x=x_Y:U\to [0,1)$  such that $x^{-1}(0)=Y$ and with $x|_{U\backslash Y}$ being smooth;
				\item  $\pi_Y$ is a locally trivial fibration whose typical fiber is the (truncated) cone $C_Z=[0,1)\times Z/\!\!\sim$ over  a closed connected manifold $Z$ (here, $\sim$ means that $\{0\}\times Z$ has been collapsed into a point), with atlas $(\phi,\mathcal V)$, where each $\phi:\pi_Y^{-1}(\mathcal V)\to \mathcal V\times C_Z$ is a trivialization and the corresponding transition functions  preserve $x$ (so that, restricted to each fiber, $x\circ \phi^{-1}$ corresponds to the natural projection $C_Z\to [0,1)$).
			\end{enumerate}
		\end{definition}
		
		In particular, 	for each $x\in (0,1)$ one has a submersion $x_Y^{-1}(x) \to Y$ whose typical fiber is $Z$, so if we send $x\to 1$ we obtain a submersion $Y^\bullet\to Y$ again with typical fiber $Z$. Hence, the manifold $X\backslash U$ has a ``fibred boundary'' $Y^\bullet$ which may be viewed as the ``resolution'' of the singular locus $Y$.
		
		We assume that $x$ has been smoothly extended to $X_s$ so that $x|_{X_s\backslash U}\equiv 1$.
		Notice that
		\[
		n=b+f+1,
		\]
		where $f$ is the dimension of $Z$. 
		Due to our use of the singular Yamabe theory in \cite{akutagawa2014yamabe}, in the following we will make the key assumption
		\begin{equation}\label{dim:asump}
			\boxed{b\leq n-2 \Longleftrightarrow f\geq 1}
		\end{equation}

		We now introduce the appropriate geometric data in $X$ (more precisely, in the smooth stratum $X_s=X\backslash Y$). For simplicity we assume that $\pi_Y$ is trivial and we fix a trivializing chart providing an identification
		\[
		U\simeq Y\times C_Z\simeq Y\times \left([0,1)\times Z/\!\sim\right),
		\] 
		with $x$ corresponding to the projection onto the $[0,1)$-factor. Locally around each cone fiber over a fixed fiber of the submersion $Y^\bullet\to X$ we may introduce coordinates $(x,y,z)$, where $(y,z)$ are (local) coordinates on $Y\times Z$.
		This allows us to consider a class of adapted metrics on $X$.

		\begin{definition}\label{edge:met}
			An (incomplete) wedge metric in a stratified space $X$ as above is a Riemannian metric $g$ on its smooth stratum $X_s$
			such that:
			\begin{enumerate}
				\item there holds
				\begin{equation}\label{ie:met}
					g|_U(x,y,z)=dx^2+x^2g_Z(z)+g_Y(y)+o(x), \quad x\to 0,
				\end{equation}
				where 
				$g_Y$ and $g_Z$ are fixed metrics in $Y$ and $Z$, respectively;
				\item the distance induced by $g$ coincides with $\mathfrak d|_{X_s}$.
			\end{enumerate}
		\end{definition}
		
		\begin{remark}\label{read:simp}
			Under suitable, but still quite restrictive, assumptions we may also treat the case in which the metric $g_Z$ varies with $y\in Y$ in a smooth way. For instance, we may assume more generally that the family of varying metrics $g_Z(y,\cdot)$, is {isospectral}. For simplicity, however, we prefer to avoid this complication.
		\end{remark}
		
		\begin{definition}\label{wedge:def}
			A pair $(X,g)$ as above is a {\em wedge space} and 
			the closed Riemannian manifold $(Z,g_Z)$ is the {\em link} of the singular stratum $Y$. In the {\em cone-edge} case $f=1$, the length of the linking circle is the {\em cone angle}.
		\end{definition}
		
		We denote by $d{\rm vol}_g$ the volume element associated to $g$. We may  define Sobolev spaces $H^{k}(X,d{\rm vol}_g)$, $k\geq 0$ an integer, in the usual way (just take the closure of Lipschitz functions under the usual Sobolev norm induced by $d{\rm vol}_g$). 
		If $f\geq 3$ we recall that the {\em conformal Laplacian} of $(Z,g_Z)$ is the elliptic operator 
		\[
		\mathcal L^{f}_{g_Z}= -\Delta_{g_Z}+\frac{f-2}{4(f-1)}R_{g_Z}.
		\]
		Here and in the following, $\Delta$ stands for the Laplacian and $R$ for the scalar curvature. 
		
		\begin{definition}\label{yam:pos:def}
			We say that $(Z,g_Z)$ is {\em Yamabe positive} if $\mathcal L^{f}_{g_Z}$ is positive definite (viewed as a self-adjoint operator acting on $L^2(Z,d{\rm vol}_{g_Z})$).	
		\end{definition}
		
		\begin{remark}\label{yam:pos}
			It is known that
			Yamabe positivity is a {\em conformal} property of $g_Z$: it is equivalent to the conformal class $[g_Z]$ of $g_Z$ carrying a metric with positive scalar curvature \cite{leeparker1987}. If this is the case then the solution of the Yamabe problem for closed manifolds \cite{schoen1984conformal,leeparker1987} allows us to replace $g_Z$ by a Yamabe metric $g'_Z\in[g_Z]$, the conformal class of $g_Z$, without affecting the wedge character of the underlying metric (in fact, the corresponding wedge metrics are easily seen to be quasi-isometric to each other, with the quasi-isometry bounds depending only on bounds on the conformal factor). In particular, after a further rescaling of the link we may assume that
			$R_{g'_Z}=f(f-1)$. This quasi-isometric replacement of wedge metrics, induced by conformal deformations on the link, is crucial here since only for $g'_Z$ as a link metric we may use Theorem \ref{const:neg}, which relies on Theorem \ref{summ}, to find a wedge metric with constant negative scalar curvature to which the perturbative scheme leading to the proof of Theorem \ref{presc:ex} below may be applied. 
		\end{remark}

		We may now state our first result, which provides an existence theorem for wedge metrics with prescribed scalar curvature under appropriate assumptions on the link. 
		
		\begin{theorem}\label{presc:ex}
			Let $(X,g)$ be a wedge space whose link $(Z,g_Z)$ satisfies either $f=1$ or $f\geq 3$ and it is Yamabe positive.
			Then any smooth and bounded function which is negative  somewhere in $X_s\backslash U$ is the scalar curvature of some wedge metric on $X$.
		\end{theorem}
		
		\begin{remark}\label{donald:1}
			The Yamabe positivity of $(Z,g_Z)$ in case $f\geq 3$ relates to the already mentioned difficulty in using the singular Yamabe theory in \cite{akutagawa2014yamabe} to find a background wedge metric with constant negative scalar curvature to which the perturbative scheme in Section \ref{resc:scalar} could be applied (see Theorem \ref{const:neg} (3)).
			As mentioned in Remark \ref{yam:pos},
			after possibly passing to a Yamabe metric $g'_Z$ in the conformal class $[g_Z]$ of $g_Z$ satisfying $R_{g'_Z}=f(f-1)$, Yamabe positivity suffices to carry out the proof of Theorem \ref{presc:ex}. 
			In any case, the fact that the class of closed manifolds of dimension $f\geq 3$ carrying metrics with positive scalar curvature is stable under surgeries of co-dimension at least $3$  \cite{gromov1980classification,schoen1979structure} provides many examples of wedge spaces to which Theorem \ref{presc:ex} applies.
			On the other hand, if $f=2$ then Yamabe positivity morally corresponds to asking that the link surface $Z$ is (topologically) a sphere, precisely the case covered by Theorem \ref{const:neg} (2); see also Remark \ref{quasi:smooth}. But notice that now $(Z,g'_{Z})$ is a {\em round} sphere, which means that the original wedge manifold is quasi-isometric (in the ``conformal'' sense of Remark \ref{yam:pos}) to a {\em smooth} manifold, in which case the conclusion of Theorem \ref{presc:ex} already follows from \cite{kazdan1975existence}. Of course, this justifies the omission of  this case in Theorem \ref{presc:ex}. Finally, as it is apparent from Theorem \ref{const:neg} (1), if $f=1$ the existence of a Yamabe wedge metric with negative scalar curvature can always be taken for granted.
		\end{remark}

		\begin{remark}\label{donald:2}
			As far as essential self-adjointness of the Laplacian operator is concerned, the analytic machinery employed in Section \ref{wedge:ell:th} works fine universally (that is, with no restriction on the link) if $f\geq 3$, but it does not seem to deliver this specific mapping property 
			if $f=2$; see Remarks \ref{f:2} and \ref{f:geq:3}. Note that mere surjectivity would suffice for our purpose of extending Theorem \ref{presc:ex} to this case but this is not needed here anyway due to the fact that, as explained in Remark \ref{donald:1}, the existence of a Yamabe wedge metric forces the surface link to be a sphere, in which case the result follows from \cite{kazdan1975existence}.
			On the other hand, the cone-edge case ($f=1$) is treated here and, as expected, self-adjointness is fulfilled if the cone angle is at most $2\pi$. Again, this may always be achieved by a quasi-isometry of the underlying wedge space induced by a rescaling of the circle link (as in Remark \ref{yam:pos}), which suffices for our purposes. 			
			For other instances of the usage of this ``at most $2\pi$'' condition in the cone-edge setting, with  relevant applications to rigidity phenomena in Geometry, we refer to \cite{hodgson1998rigidity,mazzeo2011infinitesimal,donaldson2012kahler,li2019positive,cheng2021singular}.  
		\end{remark}

		The existence results in Theorem \ref{presc:ex} should be compared with our next contribution, which provides topological obstructions to the existence of wedge metrics with (strictly) positive scalar curvature on certain {\em spin} wedge spaces. For the notion of infinite $K$-area, see Definition \ref{k:area:def}.   
		
		\begin{theorem}\label{obs:res:karea}
			Let $(X^{n},g)$, $n=2k$, be a {\em spin} wedge manifold with infinite $K$-area and whose link satisfies either
			$1<f\leq n-1$ or $f=1$ and the cone angle is at most $2\pi$. 
			Then $X$ carries no wedge metric with (strictly) positive scalar curvature (in the given quasi-isometry class of wedge metrics). Also, the same conclusion holds true for $n$ odd if $X\times \mathbb S^1$ has infinite $K$-area.  
		\end{theorem}

		\begin{remark}\label{case:bd}
			The arguments presented in \cite[Sections 2 and 5.2]{de2022scalar} may be easily adapted to provide versions of the theorems above in case the wedge space carries a boundary $\partial X$ disjoint from the wedge region $U$: one gains minimality of  $\partial X$ in the analogue of Theorem \ref{presc:ex} and should require mean convexity of $\partial X$ for the analogue of Theorem \ref{obs:res:karea}.  
		\end{remark}

		\begin{remark}\label{rosen}
			The obstruction in Theorem \ref{obs:res:karea} may be thought of as being ``complementary'' to \cite[Theorem 1.3]{albin2016index} and \cite[Theorem 2.3]{botv2021positive} in a sense that we now discuss. There exist two fundamental lines of inquiry concerning obstructions to the existence of metrics of positive scalar curvature in the spin category, which in a sense reflect the quite diverse topological contributions to the Atiyah-Singer index formula for Dirac operators. In the untwisted case, the only contribution comes from the tangent bundle, which has been explored by Lichnerowicz \cite{lichnerowicz1963spineurs} to check that the $\widehat A$-genus obstructs such metrics; this line of thought has been refined by Hitchin \cite{hitchin1974harmonic}, with the corresponding obstruction coming from $KO$-theory. On the other hand, it has been shown by Gromov and Lawson \cite{gromov1980classification,gromov1980spin,gromov1983positive} that twisting the Dirac operator with almost flat vector bundles leads to a ``complementary'' obstruction for enlargeable manifolds (for example, this works for tori, whose $\widehat A$-classes are trivial).  
			Later on, Gromov \cite{gromov1996positive} was able to somehow quantify this latter proposal by  means of the notion of ``infinite $K$-area'' 
			(compare with Definition \ref{k:area:def}). Hence, whereas the obstructions in \cite{albin2016index,botv2021positive} referred to above provide versions of the Lichnerowicz-Hitchin approach in the wedge category, Theorem \ref{obs:res:karea} aligns with  Gromov's philosophy.
		\end{remark}
		
		\section{The Yamabe problem in the wedge setting}\label{yamabe:ie}
		
		A first step towards the proof of Theorem \ref{presc:ex} involves constructing a wedge metric with {\em constant negative} scalar curvature on the given wedge space. Here we explain how this result (Theorem \ref{const:neg} below) follows from the singular Yamabe theory developed in \cite{akutagawa2014yamabe}.

		Given a wedge space $(X,g)$ as in Definition \ref{wedge:def}, we denote by $[g]_{\rm w}$ the space of wedge metrics $\widetilde g$ in $X$ which are conformal to $g$ (in the sense that there exists $u:X_s\to \mathbb R$ smooth and positive such that $\widetilde g=u^{\frac{4}{n-2}}g$). The corresponding Yamabe problem asks: there exists  $\widetilde g\in[g]_{\rm w}$ with the property that its scalar curvature $R_{\widetilde g}$ is constant?
		
		As in the smooth case, this admits a variational formulation. We fix a background wedge metric $g$ in the given conformal class and  consider the quadratic form $\mathcal Q:H^{1}(X,d{\rm vol}_{g})\to \mathbb R$, 
		\begin{equation}\label{yam:func}
			\mathcal Q(u)=\int_X\left(|du|^2+c_nR_{g}u^2\right)d{\rm vol}_{g},\quad c_n=\frac{n-2}{4(n-1)}, 
		\end{equation}
		and the constraint sphere
		\[
		\mathscr B_{n^*}=\{u\in H^{1}(X,d{\rm vol}_{g});\|u\|_{{n^*}}=1\}, \quad n^*=\frac{2n}{n-2}.
		\]
		We note that  $C^\infty_{\rm cpt}(X_s)\subset H^{1}(X,d{\rm vol}_g)$ densely, which is a consequence of the dimensional assumption (\ref{dim:asump});  see \cite[Section 2.2]{akutagawa2014yamabe}. This justifies the
		integration by parts leading to the following result.
		
		\begin{proposition}\label{var:yam}
			Critical points of $\mathcal Q|_{\mathscr B_{n^*}}$ precisely correspond to (weak) solutions of 
			\begin{equation}\label{crit:pde}
			\mathcal L_{g}u=\mu u^{\frac{n+2}{n-2}}, \quad \mu\in\mathbb R,
			\end{equation}
			where
			\[
			\mathcal L_{g}:=-\Delta_{g}+c_nR_{g}
			\]
			is the {\em conformal Laplacian} (here, $\Delta_{g}$ is the Laplacian of the background metric $g$). In particular, if $u$ is smooth and further satisfies $0<c^{-1}\leq u(x,y,z)\leq c$ for some $c>0$ and $(x,y,z)\in U$ then $\widetilde g:=u^{\frac{4}{n-2}}g$ is a solution of the corresponding Yamabe problem (in the conformal class $[g]_{\rm w}$).  
		\end{proposition}
		
		The preferred way to produce critical points for $\mathcal Q|_{\mathscr B_{n^*}}$ is by minimization. Unfortunately, the Direct Method in the Calculus of Variations can not be applied here due to the 
		fact that the continuous embedding 
		\[
		H^{1}(X,d{\rm vol}_{g})\subset L^{p}(X,d{\rm vol}_{g})
		\]
		is only  compact for $p<n^*$ \cite[Proposition 1.6]{akutagawa2014yamabe}. 
		Thus, a minimizer, if it exists, must be located by alternative methods. Also, from past experience with the smooth case fully discussed in  \cite{leeparker1987}, we expect that a minimizer should exist only in case the ``total energy'' of $[g]_{\rm w}$, as measured by the {\em global Yamabe invariant}
		\begin{equation}\label{yam:const}
			\mathcal Y_{\rm glo}(X,[g]_{\rm w}):=\inf_{u\in\mathscr B_{n^*}}\mathcal Q(u),
		\end{equation}
		which is a conformal invariant of $(X,g)$,
		lies below  a certain threshold value (this is just a manifestation of the ubiquitous bubbling off phenomenon characteristic of conformally invariant problems).  A major contribution in \cite{akutagawa2014yamabe} is precisely to identify this critical threshold. 
		
		To explain this latter point we consider, for each $V\subset X$ open,
		\[
		\mathscr	Y(V)=\inf\left\{\int_V\left(|du|^2+c_nR_{g}u^2\right)d{\rm vol}_{g};u\in H^{1}_0(V\cap X_s);\|u\|_{n^*}=1\right\}.
		\]  
		In particular, by (\ref{yam:const}),
		\[
		\mathscr	Y(X)=\mathcal Y_{\rm glo}(X,[g]_{\rm w}). 
		\]
		Notice that in principle we might have $\mathscr Y(X)=-\infty$  (in other words, $\mathcal Q|_{\mathcal B_{n^*}}$ might not be bounded from below).

		\begin{definition}\label{sob:yam:loc}
			The {\em local Yamabe invariant} of $(X,[g]_{\rm w})$ is given by 
			\[
			\mathcal Y_{\rm loc}(X,[g]_{\rm w})=\inf_{x\in X}\lim_{r\to 0}\mathscr Y(B_r(x)).
			\]
		\end{definition}
		
		\begin{remark}\label{upper:loc:yam}
			There always holds $\mathcal Y_{\rm loc}(X,[g]_{\rm w})\leq \mathcal Y_n$, 
			where $\mathcal Y_n$ is the Yamabe invariant of the round metric in the unit sphere $\mathbb S^{n}$, which follows from the fact that
			\begin{equation}\label{smooth:yam}
				\lim_{r\to 0}\mathscr Y(B_r(x))=\mathcal Y_n, \quad x\in X_s.
			\end{equation}
		\end{remark}
		
		Although some progress has been made in this regard \cite{mondello2017local,akutagawa2022non}, the actual computation of the local  Yamabe invariant is notoriously hard, the reason being that,
		amazingly enough, it depends {\em globally} on the link $(Z,g_Z)$. Precisely, 
		\begin{eqnarray*}
			\mathcal Y_{\rm loc}(X,[g]_{\rm w})
			& = & \mathcal Y(\mathbb R^{b}\times C_Z,[dy^2+dx^2+x^2g_Z])\\
			& = & 
			\mathcal Y(\mathbb H^{b+1}\times Z,[g_{\rm hyp}+g_Z]),
		\end{eqnarray*}
		where $(\mathbb H^{b+1},g_{\rm hyp})$ is hyperbolic space and $\mathcal Y$ denotes the standard Yamabe invariant. 
		Nonetheless, some of its qualitative properties may be established as a consequence of certain  integrability conditions on the scalar curvature of the background wedge metric, which also imply that $\mathcal Q|_{\mathcal B_{n^{*}}}$ is bounded from below.

		\begin{theorem}\label{equal:y:s}\cite{akutagawa2014yamabe}
			Assume that either
			\begin{equation}\label{int:sc}
				R_{g}\in L^q(X_s,d{\rm vol}_{g}), \quad {\rm for}\,{\rm some}\,q>n/2,
			\end{equation}
			or
			\begin{equation}\label{int:sc:2}
				\sup_{r>0}r^{q-n}\int_{B_r(x)}|R_g|^q d{\rm vol}_g\leq C,\quad {\rm for}\,{\rm some}\,q>1\,{\rm and}\,{\rm all}\,x\in X.
			\end{equation}
			Then $\mathcal Y_{\rm loc}(X,[g]_{\rm w})>0$ and
			$\mathcal Y_{\rm glo}(X,[g]_{\rm w})>-\infty$.
		\end{theorem}
		
		The main result in \cite{akutagawa2014yamabe}, as applied to wedge spaces, yields the following criterion for the existence of minimizers for the Yamabe functional in (\ref{yam:func}).

		\begin{theorem}\label{exis:min}\cite{akutagawa2014yamabe}
			Assume that either (\ref{int:sc}) or  (\ref{int:sc:2}) holds and 
			that 
			\begin{equation}\label{aubin:cr}
				\mathcal Y_{\rm glo}(X,[g]_{\rm w})<\mathcal Y_{\rm loc}(X,[g]_{\rm w}).
			\end{equation}
			Then there exists a minimizer for $\mathcal Q|_{\mathscr B_{n^*}}$. 
		\end{theorem}  
		
		\begin{remark}\label{acm:gen}
			We stress that the theory in \cite{akutagawa2014yamabe} applies to a much larger class of singular spaces than that considered here. Also, in the smooth case, it follows from (\ref{smooth:yam}) that Theorem \ref{exis:min} reproduces the classical Aubin's criterion for the existence of minimizers \cite[Theorem 4.5]{leeparker1987}.
		\end{remark}

		In order to use Theorem \ref{exis:min} to solve the Yamabe problem for a given wedge conformal class, we must first come to grips with the integrability conditions (\ref{int:sc})-(\ref{int:sc:2})
		on the scalar curvature of the background metric $g$. 
		As already observed in \cite[Lemma 2.4]{akutagawa2014yamabe}, it follows from the asymptotic expansion
		\begin{equation}\label{rg:exp}
			R_g\sim\left(R_{g_Z}-f(f-1)\right)x^{-2}+O(x^{-1}),\quad x\to 0,
		\end{equation}
		that either (\ref{int:sc}) or (\ref{int:sc:2})  hold true if 
		the link $(Z,g_Z)$ satisfies 
		\begin{equation}\label{rig:scalar}
			{R_{g_Z}=f(f-1)}.
		\end{equation}
		As a consequence, we obtain the following criterion for the existence of Yamabe wedge metrics.
		
		\begin{theorem}\cite{akutagawa2014yamabe}\label{summ}
			Let $(X,g)$ be a wedge space whose link $(Z,g_Z)$ satisfies one of the following conditions: 
			\begin{enumerate}
				\item $f=1$;
				\item $f=2$ and $(Z,g_Z)$ is the round unit sphere;
				\item $R_{g_Z}=f(f-1)$ if $3\leq f\leq n-1$. 
			\end{enumerate} 
			Then there exists a minimizer for $\mathcal Q|_{\mathscr B_{n^*}}$ as long as (\ref{aubin:cr}) is satisfied. 
		\end{theorem}

		\begin{remark}\label{quasi:smooth}
			Under the conditions of Theorem \ref{summ} (2),  the wedge metric $g$ actually extends smoothly to $Y$ and the result reduces to the classical criterion due to Aubin \cite[Theorem 4.5]{leeparker1987}; see Remark \ref{acm:gen}.
		\end{remark}
		
		Hence, under the assumptions on the link displayed in the previous theorem, ensuring that the scalar curvature meets the appropriate integrability condition, the general solution of the Yamabe problem gets reduced to checking the validity of the Aubin-type condition (\ref{aubin:cr}), a task that, as already observed, lies beyond the current technology due to the fact that the local Yamabe invariant is mostly inaccessible. Nonetheless, since here we are interested in merely constructing a wedge metric with {\em constant negative} scalar curvature,  we may take advantage of the freedom to suitably modify the background metric to which Theorem \ref{summ} should be applied. 
		
		\begin{theorem}\label{const:neg}
			Let $(X,g)$ be a wedge space whose link $(Z,g_Z)$ satisfies one of the following conditions:
			\begin{enumerate}
				\item $f=1$;
				\item $f=2$ and $Z$ is (topologically) a sphere;
				\item $(Z,g_Z)$ is Yamabe positive if $3\leq f\leq n-1$. 
			\end{enumerate}
			Then $X_s$ carries a wedge metric with {constant negative} scalar curvature.
		\end{theorem}
		
		\begin{proof}
			We first observe that, as explained in Remark \ref{yam:pos}, if item (3) above is satisfied then we can replace $g_Z$ by a Yamabe metric whose constant scalar curvature is $f(f-1)$, so as to ensure that item (3) in Theorem \ref{summ} is met. Also, if $f=2$ we may also reduce to Theorem \ref{summ} (2) by replacing the link by the unit sphere (but recall that here the new wedge metric turns out to be smooth). 
			After these preliminaries, 
			we proceed by using (\ref{yam:func}) with $u$ a suitable constant to get 
			\[
			\mathcal Y_{\rm glo}(X,[g]_{\rm w})\leq C\int_XR_gd{\rm vol}_g.
			\]
			As in \cite[Subsection 4.32]{besse2007einstein} we can inject
			a sufficiently large amount of negative scalar curvature around some point in $X_s\backslash U$ (without further altering the wedge region) so as to make the new background metric, still denoted by $g$, to satisfy $\int_XR_gd{\rm vol}_g< 0$. Thus, $\mathcal Y_{\rm glo}(X,[g]_{\rm w})< 0$ and  Theorem \ref{summ} applies since $	\mathcal Y_{\rm loc}(X,[g]_{\rm w})>0$ and hence (\ref{aubin:cr}) is satisfied. Leading the minimizer to (\ref{crit:pde}) we immediately see that the (constant) scalar curvature of the corresponding Yamabe wedge metric is negative, as desired. 
		\end{proof}

		\begin{remark}\label{rem:reg:1}
			The conformal factor, say $u$,  obtained in Theorem \ref{const:neg} is strictly positive and remains uniformly bounded from above as $x\to 0$. Thus, the Yamabe metric so obtained is quasi-isometric to the original background wedge metric. In fact, the regularity theory in \cite[Section 3]{akutagawa2014yamabe} guarantees that $u(x,y,z)=u_0(y)+o(1)$ as $x\to 0$, where $u_0:Y\to\mathbb R$ is smooth and strictly positive. Hence, after possibly implementing a conformal deformation on $g_Y$ and a ($y$-dependent) change in the radial coordinate, the new metric still satisfies (\ref{ie:met}), so that the corresponding link metric remains independent of $y\in Y$. 
		\end{remark}
		
		\begin{remark}\label{upshot}
			It follows from the analysis in \cite[Section 2.3]{akutagawa2014yamabe} that Theorem \ref{summ} (3) holds true under the weaker assumption that the first eigenvalue $\lambda_0(\mathcal L^n_{g_Z})$ of the operator
			\[
			\mathcal L^n_{g_Z}=-\Delta_{g_Z}+c_nR_{g_Z}, \quad c_n=\frac{n-2}{4(n-1)},
			\]
			is exactly equal to $c_nf(f-1)$; indeed, it suffices to assume $\lambda_0(\mathcal L^n_{g_Z})>0$ if $f=n-1$. Since in this case there exist positive constants $A$ e $B$ depending only on $f$ and $n$ such that
			\[
			\mathcal L^f_{g_Z}=A\mathcal L^n_{g_Z}-B\Delta_{g_Z},
			\] 
			any of these assumptions on $\mathcal L^n_{g_Z}$ actually imply that $(Z,g_Z)$ is Yamabe positive, which is the requirement in Theorem \ref{const:neg} (3). Hence, as far as the conclusion of Theorem \ref{presc:ex} is concerned, nothing is gained if we use this finer result based on $\lambda_0(\mathcal L^n_{g_Z})$.
		\end{remark}

		\section{The wedge elliptic theory for the scalar Laplacian}\label{wedge:ell:th}
		
		We describe how the wedge elliptic theory applies to Laplace-type operators on a wedge space $(X,g)$ as above. 
		As hinted at in the Introduction, this involves considering $\Delta_g$ as acting on a suitable weighted Sobolev scale and then employing the powerful micro-local methods in \cite{mazzeo1991elliptic,lesch1997differential,schulze1998boundary,egorov2012pseudo}, with further developments in \cite{gil2003adjoints,gil2013closure,albin2012signature,albin2016index,albin2018hodge,albin2022cheeger}, to check that $\Delta_g$, viewed as an unbounded, symmetric operator, has good mapping properties (essential self-adjointness and Fredholmness) for a carefully chosen value of the weight.
		We emphasize that surjectivity would suffice for the application we have in mind (Theorem \ref{presc:ex}), but we have chosen to establish finer mapping properties not only because, as explained in Remark \ref{semi:lin}, they provide a precise (local) description of the space of solutions of certain semi-linear elliptic equations on $(X,g)$, but also because we intend to transplant the analysis to the Dirac operator treated in Section \ref{top:obst}, where self-adjointness and Fredholmness are crucial in applications. Also, as the statement of our ultimate goal, Theorem \ref{self:ad:lap}, makes it clear, in this section we pose {\em no} restriction on the geometry of the link $(Z,g_Z)$ if $f\geq 3$.

		We next consider the ${\rm b}$-density  
		\[
		d{\rm vol}_{\rm b}=x^{-1}dxd{\rm vol}_{g_Y}d{\rm vol}_{g_Z}
		\] 
		in the wedge region $U$ and extend it to $X_s$ in the obvious manner.  
		
		\begin{definition}\label{sob:scale:0}
			Given an integer $k\geq 0$ and a cutoff function $\varphi$ with $\varphi\equiv 1$ near $Y$ and $\varphi\equiv 0$ outside $U$, we define $\mathcal H^{k}_{\rm b}(X)$ to be the space of all distributions $u\in \mathcal D'(X_s)$   such that:
			\begin{itemize}
				\item $(1-\varphi)u$ lies in the standard Sobolev space $H^{k}(X_s,d{\rm vol}_g)$;
				\item there holds
				\[
				(x\partial_x)^j(x\partial_y)^\mu\partial_z^\nu (\varphi u)(x,y,z)\in L^2(X,d{\rm vol}_{\rm b}), \quad j+|\mu|+|\nu|\leq k.
				\]
			\end{itemize} 
		\end{definition}
		
		\begin{definition}\label{sob:scale} (Weighted Sobolev scale)
			If  $\beta\in\mathbb R$ we set
			\[
			x^\beta \mathcal H^{k}_{\rm b}(X)=\left\{v;x^{-\beta}v\in \mathcal H^{k}_{\rm b}(X)\right\}.
			\] 
		\end{definition}

		\begin{remark}\label{scale:sob}
			Using interpolation and duality we may define $x^\beta\mathcal H^{\sigma}_{\rm b}(X)$ for any $\sigma\in\mathbb R$.
			This turns out to be a Sobolev scale of Hilbert spaces. For instance, 
			\begin{equation}\label{inner:def}
				\langle u,v\rangle_{x^\beta\mathcal H^{0}_{\rm b}(X)}=\int_X x^{-2\beta}uv\, d{\rm vol}_{\rm b}. 
			\end{equation}
			In particular, 
			one has the continuous inclusion
			\[
			x^{\beta'}\mathcal H^{\sigma'}_{\rm b}(X)\subset x^\beta\mathcal H^{\sigma}_{\rm b}(X), \quad \beta'\geq\beta,\quad \sigma'\geq\sigma,
			\]
			which is compact if strict inequalities hold. Also, if $\sigma>n/2$ then
			any $u\in x^{\beta}\mathcal H^{\sigma}_{\rm b}(X)$
			is continuous in $X_s$ and satisfies $u(x)=O(x^\beta)$ as $x\to 0$.
		\end{remark}
		
		We are interested here not only in  the {bounded}, weighted Laplacian
		\begin{equation}\label{bound:t}
			\mathcal M_{g,\beta}:=-x^2\Delta_g:x^\beta\mathcal H^{\sigma}_{\rm b}(X)\to x^\beta\mathcal H^{\sigma-2}_{\rm b}(X),
		\end{equation}
		induced by the natural action of $\mathcal M_g:=-x^2\Delta_g$ on the weighted Sobolev scale, 
		but also in the unbounded Laplacian
		\begin{equation}\label{mg:unbound}
			\mathcal C_{g,\beta}:=	-\Delta_g:C^\infty_{\rm cpt}(X)\subset x^\beta\mathcal H^{0}_{\rm b}(X)\to x^\beta\mathcal H^{0}_{\rm b}(X),
		\end{equation}
		viewed as a densely defined operator. Their analysis will eventually hinge on the pair of $f$-dependent constants
		\[
		\beta_f=-\frac{f+1}{2}, \quad 	\gamma_f=\frac{1-f}{2}.
		\]
		
		\begin{proposition}\label{self-ad-delta}
			The  Laplacian $\mathcal C_{g,\beta}$ is {symmetric}  if $\beta=\beta_f$.
		\end{proposition}
		
		\begin{proof}
			Since in the wedge region $d{\rm vol}_{\rm b}$ relates to the Riemannian volume element $d{\rm vol}_g$ of the underlying wedge metric $g$ by 
			\begin{equation}\label{vol:rel}
				d{\rm vol}_{\rm b}=x^{2\beta_f}d{\rm vol}_g,
			\end{equation}
			it follows from (\ref{inner:def}) that 
			\[
			\langle \Delta_{g}u,v\rangle_{x^\beta\mathcal H^{0}_{\rm b}(X)}=\int_X 
			x^{-2(\beta-\beta_f)}v\Delta_{g}u\, d{\rm vol}_{g}, 
			\]    
			which gives the result.
		\end{proof}
		
		\begin{remark}\label{adjust}
			The assertion in Proposition \ref{self-ad-delta}	also holds true if we replace $\Delta_g$ by any operator which is formally self-adjoint with respect to $L^2(X,d{\rm vol}_{g})$ (the Dirac operator, for instance). 
		\end{remark}

		At least in the conical case ($f=n-1$) the mapping properties of 
		$\mathcal M_{g,\beta}$
		are completely determined by the spectral resolution of $\Delta_{g_Z}$ 	in a way that we now explain. We first note that for a general wedge space, 
		\[
		\mathcal M_{g,\beta}|_{U}=-\mathsf D^2_x-\left(f-1\right)\mathsf D_x-\Delta_{g_Z} - x^2\Delta_{g_Y}+o(1),
		\]
		where $\mathsf D_x=x\partial_x$ and of course the term $x^2\Delta_{g_Y}$ should be omitted in the conical case ($Y$ collapses into a point). In any case, after discarding the last two terms with an explicit dependence on $x$
		we obtain the {\em indicial operator}
		\begin{equation}\label{ind:op:def}
			\mathcal I_{\mathcal M_{g,\beta}}=-\mathsf D^2_x-\left(f-1\right)\mathsf D_x-\Delta_{g_Z}.
		\end{equation}
		Also, if we replace $\mathsf D_x$ by $\zeta\in\mathbb C$
		we get the {\em indicial symbol} of $\mathcal M_{g,\beta}$,
		\begin{equation}\label{conorm:symb}
			I_{\mathcal M_{g,\beta}}(\zeta)=-\zeta^2-(f-1)\zeta-\Delta_{g_Z},
		\end{equation}
		a one-parameter family of elliptic operators on $(Z,g_Z)$. 
		
		Let us consider the {\em spectrum} of $\Delta_{g_Z}$,  
		\[
		{\rm Spec}(\Delta_{g_Z})=\left\{\mu\in\mathbb R; \exists u\neq 0  \,{\rm satisfying}\,-\Delta_{g_Z}u=\mu u\right\}\subset [0,+\infty). 
		\]
		For each $\mu\in {\rm Spec}(\Delta_{g_Z})$, let $\{\zeta^\pm_{\mu,f}\}$ be the (real) roots of the {\em indicial equation}
		\[
		\zeta^2+(f-1)\zeta-\mu=0,
		\]
		which is obtained by equating to zero the restriction of the right-hand side of (\ref{conorm:symb}) to each eigenspace of $\Delta_{g_Z}$.
		Explicitly, the {\em indicial roots} are
		\begin{equation}\label{indicial:r}
			\zeta^\pm_{\mu,f}=\gamma_f\pm\sqrt{\gamma_f^2+\mu}. 
		\end{equation}
	
	\begin{definition}\label{def:ind:sets}
		The {\em indicial set} is 
		\[
		\mathfrak I_{\mathcal M_{g}}^f=\bigcup_{\mu\in\, {\rm Spec}(\Delta_{g_Z})}\{\zeta^\pm_{\mu,f}\},
		\]
		which is a discrete subset of $\mathbb R$ (because $Z$ is closed), and its complement, the {\em non-indicial set}, is 
		\[
		\nonind_{\mathcal M_{g}}^f=\mathbb R\backslash \mathfrak I^f_{\mathcal M_{g}},
		\]
		a countable union of bounded, open intervals. Finally, the connected component of $\nonind_{\mathcal M_{g}}^f$ corresponding to the eigenvalue $\mu=0$ is called the {\em innermost non-indicial interval}.
		\end{definition}
		
		In the conical case, we have the following well-known result.
		
		\begin{theorem}\label{con:mazzeo}
			The map (\ref{bound:t}) is Fredholm if and only if $\beta\in \nonind^{n-1}_{\mathcal M_{g}}$, with the corresponding Fredholm index remaining the same as long as $\beta$ varies in a given connected component of $\nonind^{n-1}_{\mathcal M_{g}}$.
		\end{theorem} 
		
		This follows from \cite[Theorem 4.4]{mazzeo1991elliptic}.
		An outline of its proof, along the lines of the theory developed in \cite{schulze1998boundary,lesch1997differential}, may be found in \cite{de2022scalar,de2022mapping}, albeit the weight numerics there is slightly different from ours. Also, see \cite{schrohe2001ellipticity} for the specific information regarding the dependence of the Fredholm index on $\beta$.

		Although its proof requires a somewhat delicate analysis, it is intuitively clear from this discussion that the following holds in the conical case: as $\beta$ varies in $\nonind^{n-1}_{\mathcal M_{g}}$, the Fredholm index of (\ref{bound:t}) assumes its minimal value in the innermost non-indicial interval corresponding to $\mu=0$, namely, $I_{n-1}:=\left(2-n,0\right)$.
		Also, we may compute this minimal Fredholm index by observing  that, as explained in \cite[Corollary 3.6]{de2022mapping}, (\ref{bound:t}) is essentially self-adjoint (and hence Fredholm of index $0$) if $\beta=\beta_{n-1}$; see Proposition \ref{self-ad-delta}. For the sake of comparison with the corresponding argument below in the general wedge case, we briefly reproduce this reasoning here. 
		
		We recall that the {\em minimal domain} 
		of $\mathcal C_{g,\beta}$ as defined by (\ref{mg:unbound}) is
		\[
		\mathcal D_{\rm min}(\mathcal C_{g,\beta})=\left\{
		u\in x^\beta\mathcal H_{\rm b}^0(X);\exists\{u_n\}\subset C_{\rm cpt}^\infty(X_s);u_n\stackrel{\mathcal H_{\rm b}^0}{\longrightarrow} u, \Delta_gu_n\,{\rm is}\,\mathcal H_{\rm b}^0-{\rm Cauchy}
		\right\},
		\]
		whereas its {\em maximal domain} is 
		\[
		\mathcal D_{\rm max}(\mathcal C_{g,\beta})=\left\{
		u\in x^\beta\mathcal H_{\rm b}^0(X);\Delta_g u\in  x^\beta\mathcal H_{\rm b}^0(X)
		\right\}.
		\]
		It is known that $	\mathcal D_{\rm min}(\mathcal C_{g,\beta})\subset 	\mathcal D_{\rm max}(\mathcal C_{g,\beta})$ and that 
		(closed) sub-spaces of 
		\[
		\mathcal Q(\mathcal C_{g,\beta}):=\mathcal D_{\rm max}(\mathcal C_{g,\beta})/\mathcal D_{\rm min}(\mathcal C_{g,\beta})
		\]
		correspond to closed extensions of $\mathcal C_{g,\beta}$.
		The key observation now is the following fact whose proof may be found in \cite[Section 1.3]{lesch1997differential}, but again with a different weight numerics:
		\begin{itemize}
			\item $\mathcal Q(\mathcal M_{g,\beta})$ is formed by contributions coming from the {\em finite} set
			\[
			\mathfrak I^{n-1}_{\beta}:=\mathfrak I^{n-1}_{\mathcal M_{g}}\cap (\beta,2+\beta).
			\] 
			In particular,  if $\mathfrak I^{n-1}_\beta=\emptyset$ then $\mathcal Q(\mathcal C_{g,\beta})=\{0\}$ and  $\mathcal C_{g,\beta}$ has a unique closed extension.
		\end{itemize}
		
		We now remark that $(\beta_{n-1}, 2+\beta_{n-1})\subset I_{n-1}$ and hence  $\mathfrak I^{n-1}_{\beta_{n-1}}=\emptyset$  if $n>4$. Thus, $\mathcal C_{g,\beta_{n-1}}$ has a unique closed extension which is self-adjoint (and hence Fredholm of index $0$) because it is symmetric by Proposition \ref{self-ad-delta}. The next result then follows from the discussion above and Theorem \ref{con:mazzeo}.
		
		\begin{theorem}\label{map:lap:cr:con}
			Let $(X,g)$ be a conical space satisfying $n>4$. 
			Then the unbounded symmetric map $\mathcal C_{g,\beta_{n-1}}$  in (\ref{mg:unbound}) is essentially self-adjoint. Moreover, 
			the bounded map (\ref{bound:t}) is Fredholm of index $0$ as long as $\beta\in I_{n-1}$. 
		\end{theorem}
		
		As explained in \cite[Section 2]{de2022scalar}, at least if $n> 4$ this information suffices to carry out the proof of Theorem \ref{presc:ex} in this conical case (for more details of the argument together with a checking on how the missing case $n=4$ may be recovered by an alternative method, we refer to the proof of Theorem \ref{presc:ex} below). 
		
		We now turn to the (non-conical) purely wedge case ($1\leq f<n-1$). Here, things quickly get complicated because the indicial decomposition $\mathbb R=\mathfrak I^f_{\mathcal M_{g}}\sqcup\nonind^f_{\mathcal M_{g}}$  fails to sharply determine
		the mapping properties of (\ref{bound:t}), as no analogue of Theorem \ref{con:mazzeo} is expected to hold. Naturally enough, now the singular stratum $(Y,g_Y)$ also plays a role by contributing an extra term to the indicial operator (\ref{ind:op:def}), so as to form the so-called {\em normal operator} $\mathcal N_{\mathcal M_g}$, 
		which may be identified to $-t^2\Delta_{g_c}$, where $g_c=dt^2+t^2g_Z+\delta_u$, $\delta_u=|du|^2$, is the natural metric on $(0,+\infty)_{t}\times Z\times\mathbb R^b_u$, viewed as the tangent cone arising from the one parameter family of dilations $\mathcal T_\rho(y_0)(x,y,z)=(\rho x, y_0+\rho(y-y_0),z)$, $y_0\in Y$, as $\rho\to+\infty$. Thus,  
		\[
		\mathcal N_{\mathcal M_g}=-
		\mathsf D^2_{t}-\left(f-1\right)\mathsf D_{t}-\Delta_{g_Z} -t^2\Delta_{\delta u}. 
		\]
		
		Comparison with the indicial operator in (\ref{ind:op:def}) shows that  the extra term in 
		$\mathcal N_{\mathcal M_g}$ spoils the invariance under dilations in $t$, which reveals an essential departure from the conical case, but notice that both invariances under dilations in $(t,u)$ and translations in $u$ are retained. As in \cite{mazzeo1991elliptic} we explore this by
		first Fourier transforming in the $u$-direction with $\xi\in T^*\mathbb R^b_u$ as dual variable. We next set $\vartheta=\xi/|\xi|_{\delta_u}\in S^*\mathbb R^b_u$, the spherical conormal bundle, and $s=|\xi|_{\delta_u}t$ so as to obtain the equivalent ``Bessel-type'' normal operator
		\[
		{\mathcal B}_{\mathcal M_g}(\vartheta)=-
		\mathsf D^2_s-\left(f-1\right)\mathsf D_s -\Delta_{g_Z}+ s^2|\vartheta|_{\delta_u}^2.
		\]
		Note that the explicit dependence on $\vartheta$ is illusory since $|\vartheta|_{\delta_u}=1$. Thus, 
		$	{\mathcal B}_{\mathcal M_g}$ acts on functions in
		$\mathcal C_{Z}:=[0,+\infty)_s\times Z$, the (infinite) cone over $Z$ endowed with the conical metric $ds^2+s^2g_Z$. 
		
		To proceed, consider the spaces
		\[
		\mathcal H^{\sigma,\beta,l}(\mathcal C_{Z})=\{u\in \mathcal D'(\mathcal C_{Z});\phi u\in s^\beta \mathcal H_{\rm b}^\sigma(\mathcal C_{Z}), (1-\phi)u\in s^{-l}H^\sigma(\mathcal C_{Z})\},
		\]
		where $\phi\in C^\infty_{\rm cpt}(\mathcal C_{Z})$ with $\phi=1$ near $s=0$. Thus, we may view ${\mathcal B}_{\mathcal M_g}$ as an operator
		\begin{equation}\label{norm:bd}
			{\mathcal B}_{\mathcal M_g}:\mathcal H^{\sigma,\beta,l}(\mathcal C_{Z})\to \mathcal H^{\sigma-2,\beta,l}(\mathcal C_{Z}),
		\end{equation} 
		whose mapping properties are closely tied to Fredholmness for (\ref{bound:t}). Indeed, the next criterion follows from the general theory developed  in \cite{mazzeo1991elliptic,schulze1998boundary,egorov2012pseudo}.
		
		\begin{theorem}\label{fred:wedge}
			The map (\ref{bound:t}) is Fredholm provided $\beta \in \nonind^f_{\mathcal M_g}$ and (\ref{norm:bd}) is invertible (for such $\beta$). If (\ref{norm:bd}) is only known to be injective then (\ref{bound:t}) is semi-Fredholm (that is, it has a closed range and is essentially injective in the sense that its kernel is finite dimensional).
		\end{theorem}
		
		Thus, establishing Fredholmness of (\ref{bound:t}), which should be thought of as a preliminary step in probing its mapping properties, gets reduced to finding a range of weights for which invertibility of (\ref{norm:bd}) holds true. 
		Note that ${\mathcal B}_{\mathcal M_g}$ is {\em elliptic} as in \cite[Definition 5.3]{mazzeo1991elliptic} with the {\em same} indicial operator as $\mathcal M_g$. Thus, \cite[Lemma 5.5]{mazzeo1991elliptic} implies that
			the normal operator in (\ref{norm:bd}) is Fredholm provided $\beta \in \nonind^f_{\mathcal M_g}$. 
		It is known that the index of $\mathcal B_{\mathcal M_g}$ does not depend on the pair $(\sigma,l)$ and remains the same as long as $\beta$ varies in a given connected component of $\nonind^f_{\mathcal M_g}$.  
		Similarly, by  \cite[Corollary 5.7]{mazzeo1991elliptic} the kernel of $\mathcal B_{\mathcal M_g}$ does not depend on $(\sigma,l)$ as well, although it might change when $\beta$ crosses $\mathfrak I^f_{\mathcal M_g}$. Fortunately, we shall see that this kernel turns out to be trivial in an appropriate range of weights.
		
		\begin{proposition}\label{norm:inj}
			The normal operator $\mathcal B_{\mathcal M_g}$ in (\ref{norm:bd}) is injective if $\beta>1-f$, $f>1$. 
		\end{proposition} 
		
		\begin{proof}
			By separation of variables, we find that a solution $w\in \mathcal H^{\sigma,\beta,l}(\mathcal C_{Z})$ of the  homogeneous equation 
			${\mathcal B}_{\mathcal M_g}w=0$  
			decomposes as 
			\[
			w=\sum_{\mu\in {\rm Spec}(\Delta_{g_Z})} w_\mu,  \quad w_\mu\in {\mathcal H_{\mu}^{\sigma,\beta,l}(\mathcal C_{Z})},
			\]
			where ${\mathcal H_{\mu}^{\sigma,\beta,l}(\mathcal C_{Z})}\subset {\mathcal H^{\sigma,\beta,l}(\mathcal C_{Z})}$ selects the eigenspace of $\Delta_{g_Z}$ associated to $\mu$. Thus, it suffices to check that each $w_\mu$ vanishes. Since  ${\mathcal B}_{\mathcal M_g}w_\mu=0$, $w_\mu$
			can be expressed
			in terms of the modified Bessel functions $\mathcal I_\nu(s)$ and $\mathcal K_\nu(s)$, where 
			\[
			\nu=\nu_{\pm}=\pm\sqrt{\gamma_f^2+\mu}.
			\]
			Precisely,
			\[
			w_\mu(s)=s^{\gamma_f}\left(c_1\mathcal I_\nu(s)+c_2\mathcal K_\nu(s)\right),
			\]
			where $c_i$ is a constant. Since $\nu\neq 0$, the asymptotical behavior of these Bessel functions are as in the table below.
			\begin{center}
				\begin{tabular}{ |c|c|c| } 
					\hline
					& $s\to 0$ & $s\to +\infty$ \\ 
					\hline
					$\mathcal I_\nu(s)$ & $\sim s^{\nu}$ & $\sim e^s/s$ \\ 
					\hline
					$\mathcal K_\nu(s)$ & $\sim s^{-|\nu|}$ & $\sim e^{-s}/\sqrt{s}$ \\ 
					\hline
				\end{tabular}
			\end{center}
			The exponential growth of $\mathcal I_\nu$ at infinity forces $c_1=0$, whereas the exponential decay of $\mathcal K_\nu$ poses no restriction on $c_2$. Since $w_\mu(s)\sim s^{\gamma_f-|\nu|}$ as $s\to 0$ and 
			\[
			\gamma_f-|\nu|\leq \gamma_f-|\gamma_f|=2\gamma_f=1-f,
			\]
			the result follows.
		\end{proof}

		Although injectivity of the normal operator suffices for our purposes as it provides, via Theorem \ref{fred:wedge}, a {\em left} generalized inverse for $\mathcal M_{g,\beta}$ to which the argument in the proof of Theorem \ref{self:ad:lap} may be applied, it is known that the normal operator is surjective  in the innermost non-indicial interval $I_f$, at least if $f>1$. 
		Since this information is not used in the sequel, we omit its proof.

		\begin{theorem}\label{norm:surj}
			The normal operator (\ref{norm:bd}) is surjective if $\beta\in I_f$, $f>1$.
		\end{theorem}

		In view of Theorem \ref{fred:wedge}, it follows from the various results established above that (\ref{bound:t}) is Fredholm provided $\beta\in I_f$, $f>1$ (if we ignore Theorem \ref{norm:surj}, it is at least semi-Fredholm and essentially injective). Unfortunately, this does {\em not} suffice
		to proceed as in the conical case in order to detect self-adjoint extensions starting from $\mathcal Q(\mathcal C_{g,\beta})$; see \cite[Section 2.5]{mazzeo2012analytic} for a discussion of the issues involved. 
		We may, however, combine this information with the reasoning in \cite[Section 1]{albin2022cheeger}, which by its turn is based on arguments in \cite{gil2003adjoints,albin2012signature,albin2016index,albin2018hodge}, to check that in most cases the Laplacian $\mathcal C_{g,\beta_f}$ in (\ref{mg:unbound}), which is symmetric by Proposition \ref{self-ad-delta}, has good mapping properties. As applied to our context, the idea is that the existence of a (left) generalized inverse for (\ref{bound:t}) with $\beta\in(\beta_f,2+\beta_f)$ leads to an extra regularity in the weighted Sobolev scale for elements of $\mathcal D_{\rm max}(\mathcal C_{g,\beta_f})$ which forces them to actually belong to $\mathcal D_{\rm min}(\mathcal C_{g,\beta_f})$ by  characterizations of this latter space appearing in \cite[Proposition 3.6]{gil2003adjoints} and \cite[Theorem 4.2]{gil2013closure}.

		\begin{theorem}\label{self:ad:lap}
			Let $(X,g)$ be a wedge space satisfying either $3\leq f\leq n-1$ or $f=1$ and  the cone angle is at most $2\pi$.
			Then the Laplacian $\mathcal C_{g,\beta_f}$ defined in (\ref{mg:unbound}) is 
			essentially self-adjoint. Moreover, letting $\mathcal D_{g,\beta_f}$ be the domain of this self-adjoint extension,
			the map
			\begin{equation}\label{mg:fred}
				\mathcal C_{g,\beta_f}+\lambda:\mathcal D_{g,\beta_f}\to x^{\beta_f}\mathcal H^0_{\rm b}(X)
			\end{equation}
			is Fredholm of index $0$ for any $\lambda\in\mathbb R$.
		\end{theorem} 
		
		\begin{proof}
			Let us initially assume that $f>3$ so that both $\beta_f$ and $2+\beta_f$ lie in $I_f$. 
			We have seen that $\mathcal M_{g,\beta}$ defined in (\ref{bound:t}) is semi-Fredholm and essentially injective provided $\beta\in I_f$. In particular, there exists a {\em left} generalized inverse $\mathcal G:x^\beta\mathcal H^{0}_{\rm b}(X)\to x^\beta\mathcal H^{2}_{\rm b}(X)$ for $\mathcal M_{g,\beta}$, which means that
			\[
			\overline{\mathcal G}\Delta_g={\rm Id}-\Pi_{\beta}:x^\beta\mathcal H^{2}_{\rm b}(X)\to x^\beta\mathcal H^{2}_{\rm b}(X), \quad \beta\in I_f,
			\]
			where $\overline{\mathcal G}=-\mathcal G x^2$ and $\Pi_{\beta}$ is the projection onto the kernel of $\mathcal M_{g,\beta}$. 
			Now take $u\in \mathcal D_{\rm max}(\mathcal C_{g,\beta_f})$ so that  $u=\overline{\mathcal G}\Delta_gu+\Pi_{\beta}u$.
			From the diagram
			
			\begin{tikzcd}[row sep=huge]
				&
				& 	x^{2+\beta_f}\mathcal H^0_{\rm b}(X) \arrow[r, hookrightarrow] &	x^{\beta}\mathcal H^0_{\rm b}(X) \arrow[r,"{\mathcal G}"]
				& x^{\beta}\mathcal H^2_{\rm b}(X) \\
				& 	\mathcal D_{\rm max}(\mathcal C_{g,\beta_f}) 
				\arrow[ur,"-x^2\Delta_{g}", bend right=0]
				\arrow[urr,bend right=10,"-x^2\Delta_{g}"]
				\arrow[urrr, bend right=15,"\overline{\mathcal G} {\Delta}_{g}"]
				&
				&
			\end{tikzcd}

			\noindent
			where the inclusion in the upper row comes from the Sobolev embedding in Remark \ref{scale:sob}, we see that $\overline{\mathcal G}\Delta_gu\in  x^{\beta}\mathcal H^2_{\rm b}(X)$. Also, since $\mathcal M_{g,\beta}\Pi_\beta u=0$ one expects that $\Pi_\beta u$
			is much more regular than it appears. Indeed, it already follows from \cite[Corollary 3.24]{mazzeo1991elliptic} that $\Pi_\beta u\in x^{\beta_f}\mathcal H^\sigma_{\rm b}(X)$, $\sigma\geq 0$. With some more work we may infer that
			$\Pi_{\beta}u\in x^{\beta}\mathcal H^2_{\rm b}(X)$ as well (this relies on
			the analysis in \cite[Section 7]{mazzeo1991elliptic} and uses that  no indicial root lies in the interval $(\beta_f,\beta]$, where the normal operator is already known to be injective).
			Hence, by setting $\varepsilon=2+\beta_f-\beta$
			we obtain the inclusion
			\begin{equation}\label{enjoy:ref}
				\mathcal D_{\rm max}(\mathcal C_{g,\beta_f})\subset \bigcap_{0<\varepsilon<2} x^{2+\beta_f-\varepsilon}\mathcal H^2_{\rm b}(X),
			\end{equation}
			that is, elements of $\mathcal D_{\rm max}(\mathcal C_{g,\beta_f})$ enjoy an extra ``weighted regularity''.
			We now claim that  this leads to 
			\begin{equation}\label{still:gkm}
				\mathcal D:=
				\bigcap_{0<\varepsilon<2} x^{2+\beta_f-\varepsilon}\mathcal H^2_{\rm b}(X)
				\subset 
				\mathcal D_{\rm min}(\mathcal C_{g,\beta_f}).
			\end{equation}  
			These inclusions imply 
			$\mathcal D_{\rm max}(\mathcal C_{g,\beta_f})\subset 	\mathcal D_{\rm min}(\mathcal C_{g,\beta_f})$, that is, $\mathcal C_{g,\beta_f}$ is essentially self-adjoint. 
			
			To prove (\ref{still:gkm}) 
			we will make use of the well-known fact that, since $\mathcal C_{g,\beta_f}=-\Delta_g$ is symmetric by Proposition \ref{self-ad-delta}, $u\in \mathcal D_{\rm max}(\mathcal C_{g,\beta_f})$ is in $\mathcal D_{\rm min}(\mathcal C_{g,\beta_f})$ if and only if 
			\[
			\left<\Delta_gu,v\right>_{x^{\beta_f}\mathcal H^0_{\rm b}(X)}=\left<u,\Delta_gv\right>_{x^{\beta_f}\mathcal H^0_{\rm b}(X)},
			\]
			for any $v\in\mathcal D_{\rm max}(\mathcal C_{g,\beta_f})$; compare with \cite[Lemma 5.12]{albin2012signature}.
			If $u\in\mathcal D$ set $u_n:=x^{1/n}u$, $n\in\mathbb N$, so that $u_n\in x^{2+\beta_f}\mathcal H^2_{\rm b}(X)$. 
	 Also, for any $\varepsilon\in (0,2)$,
			\[
			u_n\to u \quad {\rm in} \quad {x^{2+\beta_f-\varepsilon}\mathcal H^2_{\rm b}(X)},
			\]
			which gives
			\begin{equation}\label{wh:giv} 
			x^{\varepsilon}\Delta_gu_n\to x^{\varepsilon}\Delta_gu
			\quad {\rm in} \quad
			{x^{\beta_f}\mathcal H^0_{\rm b}(X)}.
			\end{equation}
			Moreover, since $\varepsilon\mapsto 2-\varepsilon$ is a bijection of $(0,2)$, besides (\ref{wh:giv}) we also have  
			$x^{-\varepsilon}\mathcal D_{\rm max}(\mathcal C_{g,\beta_f})\subset x^{\beta_f}\mathcal H^2_{\rm b}(X)$. 
			Hence, if $v\in \mathcal D_{\rm max}(\mathcal C_{g,\beta_f})$ we get
			\begin{eqnarray*}
				\left<\Delta_gu_n,v\right>_{x^{\beta_f}\mathcal H^0_{\rm b}(X)} 
				& = & \left<x^{\varepsilon}\Delta_gu_n,x^{-\varepsilon}v\right>_{x^{\beta_f}\mathcal H^0_{\rm b}(X)}\\
				& \to  & \left<x^{\varepsilon}\Delta_gu,x^{-\varepsilon}v\right>_{x^{\beta_f}\mathcal H^0_{\rm b}(X)}\\
				& = & \left<\Delta_gu,v\right>_{x^{\beta_f}\mathcal H^0_{\rm b}(X)}.
			\end{eqnarray*}
			Also, 
			\[
			\left<\Delta_gu_n,v\right>_{x^{\beta_f}\mathcal H^0_{\rm b}(X)}=	\left<u_n,\Delta_gv\right>_{x^{\beta_f}\mathcal H^0_{\rm b}(X)}\to
			\left<u,\Delta_gv\right>_{x^{\beta_f}\mathcal H^0_{\rm b}(X)},
			\]
			where in the first step we used that $u_n\in \mathcal D_{\rm min}(\mathcal C_{g,\beta_f})$ and in the last one that $u_n\to u$ in 
			$x^{\beta_f+\varepsilon}\mathcal H^2_{\rm b}(X)\subset x^{\beta_f}\mathcal H^0_{\rm b}(X)$.
			Thus, 
			\[
			\left<\Delta_gu,v\right>_{x^{\beta_f}\mathcal H^0_{\rm b}(X)}=\left<u,\Delta_gv\right>_{x^{\beta_f}\mathcal H^0_{\rm b}(X)},
			\]
			which means that $u\in \mathcal D_{\rm min}(\mathcal C_{g,\beta_f})$ and hence proves (\ref{still:gkm}).
			The fact that  (\ref{mg:fred}) is Fredholm now follows from standard arguments based on the compact inclusion $\mathcal D_{g,\beta_f}=\mathcal D_{\rm max}(\mathcal C_{g,\beta_f})\subset x^{\beta_f}\mathcal H^0_{\rm b}(X)$, which follows from (\ref{enjoy:ref}) and Remark \ref{scale:sob}. Finally, using that $I_{3}=(-2,0)=(\beta_3,2+\beta_3)$, the limiting case $f=3$ may be treated by the same argument since (\ref{enjoy:ref}) still holds true even if both $\beta_f$ and $2+\beta_f$  are indicial roots.
			
			We now consider the cone-edge case ($f=1$), in which the link is a circle of length $L$ (the cone angle). It follows from (\ref{indicial:r}) that $\zeta^\pm_{0,1}=0$, so the innermost non-indicial interval disappears. If 
			we assume that $L<2\pi$ then both $\beta_1=-1\in I_1^-:=(-2\pi/L,0)$ and $2+\beta_1=1\in I_1^+:=(0,2\pi/L)$ are non-indicial; note that $I^\pm_1$ now jointly play the role of ``innermost'' non-indicial intervals. 
			Using that $\mathcal K_0(s)\sim -\log(s/2)$ as $s\to 0$, Bessel asymptotics implies that the corresponding normal operator is injective for $\beta>-{2\pi/L}$.
			Hence, we may invoke Theorem \ref{fred:wedge}
		 to ensure that $\mathcal M_{g,\beta}$ is semi-Fredholm and essentially injective for any non-indicial $\beta>-2\pi/L$, which in particular provides a left generalized inverse for $\mathcal M_{g,\beta}$,  $\beta\in(-1,1)\backslash\{0\}$, to which 
			the argument in the previous paragraph may be applied. 
			Thus, essential self-adjointness holds provided $L\leq 2\pi$. 
		\end{proof}

		\begin{remark}\label{hart}
			If we further assume that $f>3$, so that $2+\beta_f$ is not an indicial root,
			then \cite[Theorem 4.2]{gil2013closure} actually identifies $\mathcal D_{g,\beta_f}$,
			the domain of the unique self-adjoint extension of $\mathcal C_{g,\beta_f}$, to $x^{2+\beta_f}\mathcal H^2_{\rm b}(X)$; see also \cite[Theorem 1.1]{hartmann2018domain} for another proof of this result which is functional analytical in nature and altogether avoids the use of the (so pervasive!) micro-local artillery. But we insist that this extra piece of information is not needed here, so we content ourselves with the formulation of Theorem \ref{self:ad:lap}, which suffices for our purposes.
		\end{remark}

		\begin{remark}\label{f:2}
			An obvious reason why the method above does not directly apply to the case $f=2$ is that  $\beta_2=-3/2$  lies far away from the range $\beta>-1$ where the normal operator is known to be injective, so that (\ref{enjoy:ref}) is unavailable (note that the innermost non-indicial interval here is $(-1,0)$). A similar difficulty appears in the cone-edge case with cone angle $L>2\pi$. 
			A possible way to circumvent this, equally convenient for our purposes, consists in determining a   
			range of weights where 
			$\mathcal C_{g,\beta}$ would be surjective, but we will not pursue this  because, as pointed out in Remarks \ref{donald:1} and \ref{donald:2}, those cases are not   really needed here.
		\end{remark}
		
		\begin{remark}\label{spec:disc} 
			The compact inclusion
			$\mathcal D_{g,\beta_f}\subset x^{\beta_f}\mathcal H^0_{\rm b}(X)$ also implies  that the spectrum of (\ref{mg:fred}) is discrete with each eigenspace of finite multiplicity.
		\end{remark}
		
			\begin{remark}\label{f:geq:3}
			Instead of $\mathcal M_{g}$, we may  work with the
			operator
			\begin{eqnarray*}
				\widetilde{\mathcal M}_{g}
				& = & x^{-\beta_f}\mathcal M_{g}x^{\beta_f}\\
				& = &-\mathsf D_x^2+2\mathsf D_x+\beta_f(2+\beta_f)-\Delta_{g_Z}-x^2\Delta_{g_Y}+o(1),
			\end{eqnarray*}
			acting
			on the weighted Sobolev scale based on the wedge volume element
			$d{\rm vol}_{g}$, so that the ``innermost'' indicial roots associated to $\mu=0$ are $2+\beta_f$ and $-\beta_f$. Since the ``symmetry weight'' now occurs ate $\beta=0$ and $f\geq 3$ implies that $(0,2)$ has no indicial root, in this case Theorem \ref{self:ad:lap} follows from \cite[Proposition 1.5]{albin2022cheeger}. 
		\end{remark}
		
		\begin{remark}\label{sing:anal:comp}
			It is instructive to compare the analysis above with  the {\em complete} edge case, in which instead of (\ref{ie:met}) we take 
			\[
			g|_U(x,y,z)=x^{-2}\left(dx^2+x^2g_Z(z)+g_Y(y)+o(x)\right). 
			\] 
			This includes the conformally compact, asymptotically hyperbolic case with $[g_Y]$ as conformal boundary (when $Z$ collapses into a point). Here we are interested in establishing the mapping properties of the operator $\mathcal P=-\Delta_g+\lambda$, where now $\lambda$ is a smooth function on $X_s$ converging to a fixed value $\lambda_*\in\mathbb R$ as one approaches $Y$. As instances of geometric problems where this kind of question arises we mention 
			the {singular} Yamabe problem \cite{mazzeo1991conformally,mazzeo1996construction}
			and the rigidity of hyperbolic space in the context of the AdS/CFT correspondence \cite{qing2003rigidity}. As usual, we work with the appropriate weighted Sobolev scale $x^\delta\mathcal H^\sigma_{\rm e}(X)$ based on the density $x^{-1}dxd{\rm vol}_{g_Y}d{\rm vol}_{g_Z}$. To simplify matters, let us assume that $b\geq 1$ and $b^2/4+\lambda_*>0$. Since
			\[
			\mathcal P|_U=-\mathsf D_x^2+b\mathsf D_x-\Delta_{g_Z}-x^2\Delta_{g_Y}+\lambda+o(1),
			\] 
			the indicial symbol is 
			\[
			I_{\mathcal P}(\zeta)=-\zeta^2+b\zeta-\Delta_{g_Z}+\lambda_*, 
			\]
			so that the ``innermost'' non-indicial interval is $\widehat I=(\delta_-,\delta_+)$, where 
			\[
			\delta_{\pm}=\frac{b}{2}\pm\sqrt{\left(\frac{b}{2}\right)^2+\lambda_*}.
			\]
			Also, a computation shows that the normal operator $\mathcal B_{\mathcal P}$ may be identified to $-\Delta_{\widehat g}+\lambda_*$, where $\widehat g=g_{\rm hyp}+g_Z$ is the product metric on $\mathbb H^{b+1}\times Z$.
			Now, if $u\in x^\delta\mathcal H^0_{\rm e}(X)$, $\delta>\delta_-$, satisfies $\mathcal B_{\mathcal P}u=0$ then $u\in x^{b/2}\mathcal H^0_{\rm e}(X)=L^2(X,d{\rm vol}_{g})$,
			so the invertibility of $\mathcal B_{\mathcal P}$ in $\widehat I$ boils down to checking that the $L^2$ spectrum of $(\mathbb H^{b+1}\times Z,\widehat g)$ is disjoint from $(-\infty,b^2/4)$ \cite{mckean1970upper}. It then follows from the micro-local analysis in \cite{mazzeo1991conformally} that $\mathcal P$ is Fredholm of index $0$ in this range of weights; see also \cite{ma1991laplacian} for a potential theoretic approach to this same result and note that \cite{lee2006fredholm} also treats the conformally compact case by  ``low-tech'' methods. But notice that here the vanishing of the index is readily guessed given that the essential self-adjointness of $\mathcal P$ at $\delta=b/2\in\widehat I$ is a classical accomplishment \cite{gaffney1951harmonic} (recall that $g$ is complete). In particular, the analogue of the delicate analysis leading to Theorem \ref{self:ad:lap} is not needed.
		\end{remark}

		\section{An existence result: the proof of Theorem \ref{presc:ex}}\label{resc:scalar}
		
		We present here the proof of Theorem \ref{presc:ex}. 
		We start with a result confirming that  the prescription problem for the scalar curvature of wedge metrics is invariant under diffeomorphisms of bounded distortion; compare with \cite[Theorem 2.1]{kazdan1975existence} and \cite[Proposition 2.2]{de2022scalar}. 
		
		\begin{proposition}\label{conf:dist}
			Let $\phi, \phi'\in C^0(X_s)\cap L^{\infty}(X_s)\hookrightarrow x^\beta\mathcal H^0_{\rm b}(X)$, $\beta\leq 0$. If ${\rm min}\,\phi<\phi'<{\rm max}\,\phi$ then for any $\varepsilon>0$ there exists a diffeomorphism $\Psi:X_s\to X_s$ of bounded distortion (in particular, preserving the quasi-isometry class of wedge metrics) such that $\|\phi\circ\Psi-\phi'\|_{x^\beta\mathcal H^0_{\rm b}(X)}<\varepsilon$.
		\end{proposition}

		Now let $(X,g)$ be  a wedge space as in Theorem \ref{presc:ex}. 	By Theorem \ref{const:neg} and Remarks \ref{quasi:smooth} and \ref{rem:reg:1}, we may assume that $f\neq 2$ and $R_{g}= -{\bf 1}$,
		where  $\bf 1$ is the function identically equal to $1$. If $\epsilon>0$ is small enough, consider the smooth map
		\[
		A:B_\epsilon({\bf 1})\subset \mathcal D_{g,\beta_f}\to x^{\beta_f}\mathcal H^0_{{\rm b}}(X,d{\rm vol}_{\rm b})
		\]
		given by 
		\[
		A(u):=R_{u^{\frac{4}{n-2}} g}=-u^{-\alpha_n}\left(c_n^{-1}\Delta_{ g}u+u\right), \quad \alpha_n=\frac{n+2}{n-2}.
		\]
		Note that  $A({\bf 1})=-{\bf 1}$.
		
		\begin{proposition}\label{lin:isom}
			The linearization 	$\dot A_{\bf 1}:\mathcal D_{g,\beta_f}\to x^{\beta_f}\mathcal H^0_{{\rm b}}(X)$ is an isomorphism.
		\end{proposition}
		
		\begin{proof}
			A short computation gives
			\begin{equation}\label{iso:linear}
				\dot A_{\bf 1}=c_n^{-1}\left(-\Delta_{g}+\lambda_n\right), \quad \lambda_n=\frac{4c_n}{n-2}=\frac{1}{n-1},
			\end{equation}
			which is Fredholm of index $0$ by Theorem \ref{self:ad:lap} (after a possible rescaling of the circle link if $f=1$).
			Alternatively, we may use Theorem \ref{map:lap:cr:con} if $f=n-1$, $n>4$.
			Note that self-adjointness of $\mathcal C_{g,\beta_f}$ guarantees that integration by parts does not yield a contribution coming from the singular stratum. Thus, if $\dot A_{\bf 1}u=0$ we get 
			\[
			-\left\|\nabla_gu\right\|^2_{x^{\beta_f}\mathcal H^0_{\rm b}(X)}=\lambda_n\left\|u\right\|^2_{x^{\beta_f}\mathcal H^0_{\rm b}(X)},
			\]
			a contradiction unless $u=0$, which confirms injectivity of $\dot A_{\bf 1}$. The result now follows from Fredholm alternative.
		\end{proof}
		
		By the Inverse Function Theorem, there exists $\varepsilon>0$ and a neighborhood $U\subset \mathcal C_{g,\beta}$ of ${\bf 1}$ such that $A|_U:U\to B_{\varepsilon}(-{\bf 1})\subset x^{\beta_f}\mathcal H^0_{{\rm b}}(X)$ is a diffeomorphism. Also, if $F$ is the prescribed function then there exists $K>0$ such that  $K\min F<-1<K\max F$	
		and Proposition \ref{conf:dist} gives a diffeomorphism $\Psi:X_s\to X_s$ of bounded distortion such that $KF\circ\Psi\in B_{\varepsilon}(-{\bf 1})$, which implies that $KF\circ\Psi$ may be realized as the scalar curvature function of some wedge metric $\widetilde g$. Thus, $K^{1/2}(\Psi^{-1})^*\widetilde g$ is the required wedge metric (whose scalar curvature is $F$). 
		Finally, we note that this approach turns out to be effective enough to allow us to prove Theorem \ref{presc:ex} in the remaining conical case $n=4$. Indeed, it suffices to take $\beta=\beta_3$ and use that, again by Theorem \ref{self:ad:lap}, $\mathcal M_{g,\beta_3}$ as in (\ref{mg:unbound}) is essentially self-adjoint. 
		
		\begin{remark}\label{rem:reg:2}
			Again, the regularity theory in \cite[Section 3]{akutagawa2014yamabe} ensures that the conformal metric produced in Theorem \ref{presc:ex} lies in the same quasi-isometry class as the original wedge metric and hence is wedge as well; compare with Remark \ref{rem:reg:1}.
		\end{remark}

		\begin{remark}\label{semi:lin}
		Although surjectivity of (\ref{iso:linear}) would suffice to produce the metric $\widetilde g$ above via the Implicit Function Theorem, the finer isomorphism property shows that the space of conformal factors yielding wedge metrics whose scalar curvature functions are close to $-{\mathbf 1}$ may be locally parameterized by an open ball centered at ${\mathbf 1}\in\mathcal D_{g,\beta_f}$. In particular, local uniqueness for the associated non-linear problem holds. As another example of this sort of phenomenon, we note that with no restriction at all on the geometry of the link if $f\geq 3$, the perturbative method above may be used to check that 
		for any $v$ close enough to ${\bf 1}$ there is a (unique) $u$ close to ${\bf 1}$ solving the semi-linear elliptic equation
		\[
		\Delta_gu+u=vu^\alpha,\quad \alpha>1.
		\]
		\end{remark}
		
		\begin{remark}\label{holder}
			A routine procedure that we omit here allows us to 
			recast both Theorem \ref{self:ad:lap} and the proof of Theorem \ref{presc:ex} in the language of weighted H\"older spaces, which is more convenient to directly handle the regularity of solutions that we actually refrained from spelling out in the
			Sobolev analysis above.  
		\end{remark}

		\section{A topological obstruction: the proof of Theorem \ref{obs:res:karea}}\label{top:obst}
		
		We now look at topological obstructions to the existence of wedge metrics with (strictly) positive scalar curvature in a given wedge space. Our aim is to present the proof of Theorem \ref{obs:res:karea} and for this we will use a version of Index Theory as developed in \cite{albin2016index}, so we assume from now on that $X_s$ is spin. Thus, in the presence of a wedge
		metric $g$, we may consider the corresponding Dirac operator acting on the associated spin bundle
		\[
		\dirac:\Gamma(\mathbb S_X)\to \Gamma(\mathbb S_X). 
		\]  
		As before, we may define the weighted Sobolev scale $x^\beta\mathcal H_{\rm b}^{\sigma}(\mathbb S_X)$ formed by appropriate distributional spinors, so the question remains of determining for which values of $\beta$ the map
		\begin{equation}\label{dirac:bd}
			\dirac: x^\beta\mathcal H_{\rm b}^{\sigma}(\mathbb S_X)\to x^\beta\mathcal H_{\rm b}^{\sigma-1}(\mathbb S_X)
		\end{equation}
		is at least semi-Fredholm and essentially injective, as this is the first step in trying to establish good mapping properties. The argument below, leading to Theorem \ref{dirac:map:p}, adapts to the Dirac setting the approach adopted in Section \ref{wedge:ell:th} for Laplace-type operators and should be thought of as a variation of the reasoning in \cite{albin2016index}.  
		
		It follows from \cite[Lemma 2.2]{albin2016index} that, in the wedge region $U$,
		\begin{equation}\label{dirac:wed:reg}
			x{\dirac}={\mathfrak c}(\partial_x)\left(\mathsf D_x+\frac{f}{2}+{\dirac}_Z+x\dirac_Y\right)+o(x),
		\end{equation}
		where $\mathfrak c$ is Clifford product and ${\dirac}_Z$ (respectively, $\dirac_Y$) is the Dirac operator of $(Z,g_Z)$ (respectively, $(Y,g_Y)$). Replacing $\mathsf D_x$ by $\zeta$ in the right-hand side above and sending $x\to 0$, we see that the corresponding indicial symbol is 
		\[
		I_\dirac(\zeta)=\zeta+\frac{f}{2}+\dirac_Z.
		\]
		Since we aim at obstructing positive scalar curvature metrics, we may assume that $R_g|_U\geq 0$. It
		follows from (\ref{rg:exp}) that
		$R_{g_Z}\geq f(f-1)>0$ if $f>1$.  A well-known estimate \cite[Section 5.1]{friedrich2000dirac} then gives the spectral gap
		\begin{equation}\label{spec:gap}
			{\rm Spec}(\dirac_{Z})\cap \left(-\frac{f}{2},\frac{f}{2}\right)=\emptyset,
		\end{equation}
		known as the ``geometric Witt condition''.
		As explained in \cite{de2022scalar,de2022mapping}, at least in the conical case ($f=n-1$), this analysis suffices to guarantee that (\ref{dirac:bd}) is Fredholm of index $0$ whenever $\beta\in J_{n-1}:=(1-n,0)$ since (\ref{spec:gap}) prevents the appearance of indicial roots in this interval.  In the general wedge case, however, one should take into account the effects coming from the corresponding normal operator
		\[
		{\mathcal B}_\dirac(\vartheta)=\mathfrak c(\partial_s)\left({\mathsf D}_s+\frac{f}{2}+\dirac_Z\right)+s{{\bf i}}\mathfrak c(\vartheta), \quad \vartheta\in S^*\mathbb R_u^b.
		\]
		Luckily, the general mapping theory of wedge elliptic operators \cite{mazzeo1991elliptic,schulze1998boundary} tells us that the pattern to follow here is quite similar to the case of the Laplacian studied in Section \ref{wedge:ell:th}, which will allow us to easily transplant the previous analysis to the present context. Firstly, we check that when acting on the appropriate Sobolev scale on $\mathcal C_Z$, ${\mathcal B}_\dirac$ is injective for $\beta>-f$ and surjective for $\beta\in J_f:=(-f,0)$, which plays the role of the ``innermost'' non-indicial interval in this setting; see the proof of \cite[Lemma 3.10]{albin2016index}, but be aware that the indicial numerics there is different from ours essentially because their weighted Sobolev spaces are based on the volume element $d{\rm vol}_g$ instead of $d{\rm vol}_{\rm b}$; another approach to this issue appears in \cite[Section 6]{hartmann2018domain}. 
		We insist, however, that only the injectivity of ${\mathcal B}_\dirac$ on $J_f$ is really needed in the sequel, and this is an easy consequence of elementary Bessel asymptotics (as in the proof of Proposition \ref{norm:inj}).  
		Thus, (\ref{dirac:bd}) is semi-Fredholm and essentially injective for $\beta\in J_f$ by the appropriate analogue of Theorem \ref{fred:wedge}. Secondly, we combine this information (translated into the existence of a left generalized inverse
		for (\ref{dirac:bd}) in this range of weights) with the argument in the proof of \cite[Theorem 3.11]{albin2016index}, which is the Dirac version of \cite[Section 1]{albin2022cheeger}, to establish the following result.

		\begin{theorem}\label{dirac:map:p}
			Let $(X,g)$ be a spin wedge space with $R_g|_U\geq 0$. 
			If $f=1$ assume also that the cone angle is at most $2\pi$. Then  the unbounded map 
			\begin{equation}\label{dirac:self}
				\dirac:\Gamma_{\rm cpt}(X_s)\subset x^{\beta_f}\mathcal H_{\rm b}^{0}(\mathbb S_X)\to x^{\beta_f}\mathcal H_{\rm b}^{0}(\mathbb S_X)
			\end{equation}
			is	essentially self-adjoint and the corresponding self-adjoint extension is Fredholm.  
		\end{theorem}
		
		\begin{proof}
			As already advertised, this is just a matter of transplanting the analysis for the Laplacian in Section \ref{wedge:ell:th} to this Dirac setting.
			In particular, have Remark \ref{adjust} in mind and note that $2+\beta_f$ should be replaced by $1+\beta_f$, as dictated by the left-hand side of (\ref{dirac:wed:reg}).
			If $f>1$ then both $\beta_f$ and $1+\beta_f$ lie in $J_f$ and we may proceed exactly as in the proof of Theorem \ref{self:ad:lap} (where the corresponding assumption is $f>3$).
			Notice that 
			$f>1$ is also required here to justify the spectral estimate leading to (\ref{spec:gap}). Nonetheless, 
			(\ref{spec:gap})
			holds true in the cone-edge case $f=1$ if the cone angle is at most $2\pi$ and we choose the {\em bounding} spin structure on the link circle; see \cite{chou1989criteria,albin2016index} for discussions on this subtlety. In any case, with this extra information at hand, our reasoning above may also be transplanted to cover the limiting case $f=1$ since 
			$J_1=(-1,0)=(\beta_1,1+\beta_1)$ (compare again with the proof of Theorem \ref{self:ad:lap}, where the ``limiting'' case is $f=3$). 
		\end{proof}

		\begin{remark}\label{hart:2}
			As in Remark \ref{hart}, it follows from \cite{gil2013closure,hartmann2018domain} that the domain of the self-adjoint extension of $\dirac$ above is  $x^{1+\beta_f}\mathcal H^1_{\rm b}(\mathbb S_X)$, at least if $f> 1$. If $f=1$ and the cone angle is {\em strictly} less that $2\pi$ then $0=1+\beta_1$ is not an indicial root and the corresponding domain is  $\mathcal H^1_{\rm b}(\mathbb S_X)$.
		\end{remark}
		
		We now make use of the theory above to find topological obstructions to the existence of wedge metrics of positive scalar curvature.
		Let  us still denote by $\dirac$ the self-adjoint realization of (\ref{dirac:self}). If $n=2k$ is even, it is well -known that $\dirac$ splits as
		\[
		\dirac=
		\left(
		\begin{array}{cc}
			0 & \dirac^- \\
			\dirac^+ & 0
		\end{array}
		\right)
		\]
		where $\dirac^\pm$ are the realizations of the {\em chiral} Dirac operators
		\[
		\dirac^\pm:\Gamma_{\rm cpt}(\mathbb S^\pm_X)\to \Gamma_{\rm cpt}(\mathbb S^\mp_X)
		\]
		corresponding to the chiral decomposition $\mathbb S_X=\mathbb S^+_X\oplus\mathbb S^-_X$. It then follows from Theorem \ref{dirac:map:p} that these chiral Dirac operators are Fredholm and adjoint to each other, so it makes sense to consider the associated {index}
		\[
		{\rm ind}\,\dirac^+=\dim\ker\dirac^+-\dim\ker\dirac^-. 
		\] 
		More generally, if $\mathcal E$ is a Hermitian vector bundle over $X_s$ endowed with a compatible connection, we may consider the 
		{\em twisted} Dirac operator
		\[
		\dirac\otimes{\mathcal E}=
		\left(
		\begin{array}{cc}
			0 & \dirac^-\!\!\otimes{\mathcal E} \\
			\dirac^+\!\!\otimes{\mathcal E} & 0
		\end{array}
		\right)
		\]
		In general, twisting with $\mathcal E$ spoils self-adjointness since  $\dirac_{g_Z}\!\otimes\mathcal E|_Z$ should somehow appear in the expressions of  the corresponding indicial symbol and normal operator, thus compromising the analysis above. However,
		if $\mathcal E$ is {\em admissible} in the sense that $\mathcal E|_U$ is trivial (and endowed with a flat connection) then  $\dirac\otimes\mathcal E|_U$ is a sum of $r:={\rm rank}\,\mathcal E$ copies of $\dirac|_U$ and, as explained in the proof of \cite[Proposition 4.1]{de2022scalar}, Theorem \ref{dirac:map:p} and Remark \ref{hart:2} hold {\em verbatim} for $\dirac\otimes{\mathcal E}$,
		so we can define the corresponding index
		\[
		{\rm ind}\,\dirac^+\!\!\otimes{\mathcal E}=\dim\ker\dirac^+\!\!\otimes{\mathcal E}-\dim\ker\dirac^-\!\!\otimes{\mathcal E}.
		\]
		
		It turns out that this fundamental invariant can be explicitly computed  in terms of the underlying geometry by means of heat asymptotics \cite{chou1985dirac,lesch1997differential,albin2016index}. Indeed, if  $\Theta^r$ is the trivial bundle of rank $r$ and $\widetilde{\mathcal E}=\mathcal E-\Theta^r$ is the associated virtual bundle then \cite[Main Theorem]{albin2016index} tells us that 
		\begin{equation}\label{ind:albin}
			{\rm ind}\,\dirac^+\!\!\otimes{\widetilde{\mathcal E}}=r\,{\rm ind}\,\dirac^++\int_{X_s}\widehat A(TX_s)\wedge{{\rm ch}\,\widetilde{\mathcal E}},
		\end{equation}	
		where 
		\[
		{\rm ind}\,\dirac^+=\int_{X_s}\widehat A(TX_s)+\int_Y\widehat A(TY)\left(-\frac{1}{2}\eta_{\dirac_Z}+\int_Z\mathcal T\widehat A(TX_s)\right),
		\]
		$\widehat A$ is the $\widehat A$-class, $\mathcal T$ means transgression, $\eta_{\dirac_Z}$ is the eta invariant of $\dirac_Z$ and ${{\rm ch}\,\widetilde{\mathcal E}}$ is the Chern character of $\widetilde{\mathcal E}$. 
		
		According to \cite[Theorem 1.3]{albin2016index}, if $g$ has (strictly) positive scalar curvature everywhere then ${\rm ind}\,\dirac^+=0$, so that (\ref{ind:albin}) reduces to 
		\begin{equation}\label{ind:form:sp}
			{\rm ind}\,\dirac^+\!\!\otimes{\widetilde{\mathcal E}}=\int_{X_s}\widehat A(TX_s)\wedge{{\rm ch}\,\widetilde{\mathcal E}},
		\end{equation}
		for any admissible bundle $\mathcal E$. Notice that $\widehat A(TX_s)$ only contributes to the right-hand side with ``lower order'' terms, which aligns with the comments in Remark \ref{rosen}.
		We now introduce the class of wedge spaces to which Theorem \ref{obs:res:karea} applies; compare with \cite[Section 3]{de2022scalar}, where the conical case is discussed, and also with \cite{gromov1996positive}, where this notion was originally conceived in the smooth category. 
		
		\begin{definition}\label{k:area:def}
			A wedge space $(X,g)$ has {\em infinite} $K$-{\em area} if for any $\epsilon>0$ there exists an admissible, $\epsilon$-flat bundle over $X$ which is homologically nontrivial in the sense that at least one of its Chern numbers does not vanish. 
		\end{definition}
		
		A key point here is that having infinite $K$-area is a {\em quasi-isometric} property of the wedge space $(X,g)$. In any case, with (\ref{ind:form:sp}) at hand and proceeding exactly as in \cite[Section 5.1]{de2022scalar}, the proof of Theorem \ref{obs:res:karea} follows immediately.
		
		\bibliographystyle{alpha}
		\bibliography{yamabe-edge}		
	\end{document}